\documentclass[12pt]{article}
\usepackage[utf8]{inputenc}
\usepackage{amsmath}
\usepackage{amssymb}
\usepackage{amsthm}
\usepackage{algorithm}
\usepackage{algorithmicx}
\usepackage[noend]{algpseudocode}
\usepackage{fullpage}
\usepackage{graphicx}
\usepackage{hyperref}
\usepackage{color}
\usepackage[sc]{mathpazo}
\usepackage[T1]{fontenc}
\usepackage[margin=10pt,font=small]{caption}

\newtheorem{thm}{Theorem}[section]
\newtheorem{prop}[thm]{Proposition}

\theoremstyle{remark}
\newtheorem{rmk}{Remark}

\newcommand{\tdef}[1]{\textcolor{blue}{\emph{#1}}}
\newcommand{\horiz}{\begin{center}\rule{0.3\textwidth}{0.5pt}\end{center}}

\newcommand{\naturals}{\mathbb{N}_+}
\newcommand{\pred}{\texttt{pred}}

\author{Wenjie Fang \\ Univ Gustave Eiffel, CNRS, LIGM, F-77454 Marne-la-Vallée, France}

\title{Searching on the boundary of abundance for odd weird numbers}
\begin{document}

\maketitle

\abstract{Weird numbers are abundant numbers that are not pseudoperfect. Since their introduction, the existence of odd weird numbers has been an open problem. In this work, we describe our computational effort to search for odd weird numbers, which shows their non-existence up to $10^{21}$. We also searched up to $10^{28}$ for numbers with an abundance below $10^{14}$, to no avail. Our approach to speed up the search can be viewed as an application of reverse search in the domain of combinatorial optimization, and may be useful for other similar quest for natural numbers with special properties that depend crucially on their factorization.}

\horiz

\section{Introduction}

Given $N \in \naturals$, we denote by $\sigma(N)$ the sum of its divisors, and we define its \tdef{abundance}, denoted by $A(N)$, by $A(N) = \sigma(N) - 2N$. The number $N$ is \tdef{deficient} if $A(N) < 0$, \tdef{abundant} if $A(N) > 0$, and \tdef{perfect} if $A(N) = 0$. An abundant number is \tdef{pseudoperfect} (or \tdef{semiperfect}) if some subset of its proper factors sums to itself, or equivalently to $A(n)$. An abundant number is \tdef{weird} if it is not pseudoperfect. The smallest weird number is $70 = 2 \times 5 \times 7$, as $A(70) = 4$ is clearly not the sum of any subset of the factors of $70$, starting with $1, 2, 5$.

The notion of weird number was defined by Benkoski in \cite{weird-def}. Then Benkoski and Erdős proved in \cite{weird-pseudoperfect} that weird numbers have positive density among natural numbers. They also observed that, if $N$ is a weird number, then $pN$ is also weird for any prime $p > \sigma(n)$. This leads to the notion of \tdef{primitive weird numbers}, which are weird numbers without any proper factor that is also weird. Benkoski and Erdős suggested that there may be infinitely many primitive weird numbers, but there was no proof and it remains an open problem. Recently, Melfi proved in \cite{conditional-weird} that, conditional to some unsolved conjecture on gaps of consecutive primes, there are infinitely many primitive weird numbers of the form $2^k pq$, with $p$ and $q$ primes.

Another open problem also proposed in \cite{weird-pseudoperfect} is the existence of odd weird number. The author of \cite{conditional-weird} noted that a variation of their main result leads to the existence of odd weird numbers, always conditional to the same unsolved conjecture, but also to the existence of odd almost perfect numbers, \textit{i.e.}, odd numbers with abundance $-1$. This is another long-standing problem in number theory.

Given the difficulty of constructing odd weird numbers, and more generally of the study of weird numbers, we naturally turn to an experimental approach. There are some existing work on searching for large primitive weird numbers, started by Kravitz \cite{large-weird}. Recently, Amato, Hasler, Melfi and Parton \cite{primitive-with-factors} searched for primitive weird numbers with a given number of prime factors (counted with multiplicity), and found such numbers with 16 prime factors and 14712 digits. However, there is few effort specifically dedicated to the search of odd weird numbers. According to A006037 on The On-Line Encyclopedia of Integer Sequences \cite{oeis-weird}, in 2005, Robert A. Hearn checked that there is no odd weird number below $10^{17}$.

Our work consists of searching exhaustively for odd weird numbers below a certain upper bound. We obtained the following result.

\begin{thm} \label{thm:1e21}
  There is no odd weird number below $10^{21}$.
\end{thm}

By restricting the abundance, we also searched numbers with a much higher upper bound.

\begin{thm} \label{thm:1e28}
  There is no odd weird number below $10^{28}$ with abundance smaller than $10^{14}$.
\end{thm}

Both results were obtained through collaboration with the volunteer computing project yoyo@home during 2013--2015 (see \cite{ows-yoyo} for details).

Our method is quite close to that of \cite{primitive-with-factors}, searching for numbers by constructing their factorization directly. However, the search therein comes with a limited number of prime factors, while our limit is the size of the number. Furthermore, in their search, when trying to add a prime factor to a deficient number, they only consider at most a fixed number of possible primes, usually quite small (less than 10 in the given examples). Therefore, the numbers we check here is different from that of \cite{primitive-with-factors}.

As a historical remark, although our search predates \cite{primitive-with-factors}, the method itself may have been used earlier in the search of primitive abundant numbers with different properties. For instance, considering the computational power of personal computers in 2005, the effort of Hearn cited in \cite{oeis-weird} may have used the same method. Furthermore, as mentioned in \cite{primitive-with-factors}, it was proved by Dickson in \cite{dickson1} that, for any $k \in \naturals$, there is only a finite number of odd primitive non-deficient numbers with $k$ distinct prime factors. This also gives a strong incentive to search for odd weird numbers by multiplying distinct primes. As a side note, our method can be seen as a particularly simple application of reverse search \cite{reverse-search}.

The rest of the article is organized as follows. In Section~\ref{sec:tree}, we explain the tree structure we used for the searches. Then the algorithm we used is given and proved in Section~\ref{sec:algo}. Finally, implementation details are given in Section~\ref{sec:impl}.

\section{Tree structure of $\naturals$} \label{sec:tree}

We write $a \mid b$ when $b$ is a multiple of $a$. Then $(\naturals, \mid)$ is a partially ordered set (or simply \tdef{poset}) induced by the multiplicative structure of $\naturals$, with $1$ the only minimal element. By the fundamental theorem of arithmetic, every $N \in \naturals$ can be written uniquely as $N = p_1^{m_1} p_2^{m_2} \cdots p_k^{m_k}$, where $p_1 < \ldots < p_k$ are prime numbers, and $m_i > 0$ is the multiplicity of $p_i$ in $N$. For $N = p_1^{m_1} \cdots p_k^{m_k} > 1$, we define the \tdef{predecessor} $\pred(N)$ of $N$ to be
\[
  \pred(N) = p_1^{m_1} \cdots p_k^{m_k - 1}.
\]
Simply speaking, $\pred(N)$ is $N$ divided by its largest prime factor. The notion of predecessor is compatible with $(\naturals, \mid)$, in the sense that $\pred(N) \mid N$ and there is no element between $\pred(N)$ and $N$ in $(\naturals, \mid)$. The following property is trivial.

\begin{prop} \label{prop:descendants}
  Given $N \in \naturals$ with $p_*$ its largest prime factor, every $N'$ such that $\pred(N') = N$ takes the form $N' = Np$, where $p$ is a prime and $p \geq p_*$.
\end{prop}

Using $\pred(N)$, we may endow $\naturals$ with a tree structure, where $1$ is the only root, and the parent of $N$ is $\pred(N)$. We thus have an infinite tree $T$, where each number $N$ has infinitely many \tdef{children}, which are numbers $N'$ with $\pred(N') = N$. The \tdef{descendants} of $N$ are numbers whose chain of parents, obtained by applying $\pred$ iteratively, contains $N$. In other words, descendants of $N$ can be obtained by multiplying $N$ with any number of primes larger than the largest prime factor of $N$. We will perform our search of odd weird numbers on $T$, as it gives automatically the factorization for each number $N$ visited, which is essential for checking whether $N$ is weird. We will call $T$ our \emph{search tree}.

\section{Search algorithm} \label{sec:algo}

We want to search for odd weird numbers in $T$ by a depth-first search. To reduce the computation, we need to remove irrelevant parts of $T$. We already limit ourselves to a finite tree by imposing an upper bound in our search, thus search only a finite part of the infinite tree $T$. However, we may further prune the search tree $T$ using number-theoretic reasoning.

The search of odd weird number is closely related to primitive abundant numbers by the following proposition.

\begin{prop} \label{prop:primitive-weird}
  Suppose that $N$ is the smallest odd weird number, then it is also an odd primitive abundant number.
\end{prop}
\begin{proof}
  From definition, we know that $N$ is abundant. If $N$ is not a primitive abundant number, then there is some primitive abundant number $N'$ such that $N = kN'$ for some $k > 1$. By the minimality of $N$, we know that $N'$ is pseudoperfect. However, in this case, there is a subset $S$ of its proper divisors of $N'$ that sums to $N'$, meaning that the set $kS$, obtained by multiplying each element in $S$ by $k$, sums to $N = kN'$. Every element in $kS$ is also a proper divisor of $N$. Thus, $N$ is also pseudoperfect, which leads to a contradiction. We thus have our claim.
\end{proof}

By Proposition~\ref{prop:primitive-weird}, in our search on $T$, if we encounter a pseudoperfect number, then we don't need to check its descendants for odd weird numbers.

As mentioned before, We only search on numbers below a certain upper bound $M$. Such an upper bound also helps restricting the factorization of numbers to check. In \cite{abundant-prime-factor}, Iannucci gave an algorithm for the smallest abundant number not divisible by the first $k$ primes, and also its asymptotic behavior. For our need, we only need the following results.

\begin{prop}[See Table~1 in \cite{abundant-prime-factor}] \label{prop:abundant-prime-factor}
  The smallest odd abundant number not divisible by $3$ and $5$ is
  \[
    7^2 \times 11^2 \times 13 \times 17 \times \cdots \times 61 \times 67 < 2.01 \times 10^{25}.
  \]
  The smallest odd abundant number not divisible by $3$, $5$ and $7$ is
  \[
    11^2 \times 13^2 \times 17 \times 19 \times \cdots \times 131 \times 137 < 4.90 \times 10^{52}.
  \]
  The omitted parts are consecutive primes.
\end{prop}

By Proposition~\ref{prop:abundant-prime-factor} , when we search for odd weird numbers below $2.01 \times 10^{25}$ (resp. $4.90 \times 10^{52}$), we may look only at numbers whose smallest prime factor is at most $5$ (resp. $7$).

Another possibility to prune the search tree is to avoid exploring numbers that can never lead to abundant numbers within the upper bound $M$. For instance, for $M$ large enough, there is no need to explore numbers of the form $3p$ with $p$ a prime larger than $M^{1/3}/3$. It is because, by Proposition~\ref{prop:descendants}, the descendants of $3p$ take the form $3pq$ with $q$ a prime number larger than $p$, but they are clearly not abundant, and their descendants goes above $M$, again by Proposition~\ref{prop:descendants}. Similar reasoning can be extended as follows.

\begin{prop} \label{prop:barrier-formula}
  Let $N < M$ with $p_*$ its largest prime factor. Suppose that $k_*$ is the largest integer such that $N p_*^{k_*} < M$, and that
  \begin{equation} \label{eq:barrier}
    \frac{\sigma(N)}{N} \cdot \frac{p_*^{k_*}}{(p_*-1)^{k_*}} < 2.    
  \end{equation}
  Then all descendants of $N$ in $T$ smaller than $M$ are deficient.
\end{prop}
\begin{proof}
  We define $r(n) = \sigma(n)/n$. Suppose that $N' < M$ is a descendant of $N$ in $T$. We can write $N' = N q_1 q_2 \cdots q_k$, with all $q_i$'s primes with $q_i > p_*$. Here some of the $q_i$'s may be the same. We thus have $k \leq k_*$, and then
  \[
    r(N') \leq r(N) \prod_{i=1}^k \frac{q_i+1}{q_i} < r(N) \frac{p_*^{k_*}}{(p_*-1)^{k_*}}.
  \]
  The first inequality holds with repeated $q_i$ because $1+q+\ldots+q^k \leq (1+q)^k$. We note that it holds even when some $q_i$ is equal to $p_*$. The second inequality holds as $f(x) = (x+1)/x$ is greater than $1$ and strictly decreasing on $(0, +\infty)$. By Equation~\eqref{eq:barrier}, we thus have $r(N') < 2$, meaning that $N'$ is deficient.
\end{proof}

As Equation~\eqref{eq:barrier} is relatively easy to check, we may use Proposition~\ref{prop:barrier-formula} to prune the search tree efficiently. We may construct more precise conditions to further prune the search tree, but such gain seems to be minimal. Heuristically, it is because there are relatively few non-square-free primitive abundant numbers, and Proposition~\ref{prop:barrier-formula} is quite tight for square-free ones.

Using all previous results, we construct the following abstract algorithm (Algorithm~\ref{algo:weird}) we used to establish Theorem~\ref{thm:1e21}~and~\ref{thm:1e28}. We suppose the existence of the following functions:
\begin{itemize}
\item \texttt{CheckWeird(N)}: checks whether $N$ is weird, and reports if affirmative;
\item \texttt{CheckBarrier(N)}: checks the condition in Proposition~\ref{prop:barrier-formula} on $N$ and returns whether it is satisfied;
\item \texttt{NextPrime(p)}: returns the smallest prime strictly larger than $p$.
\end{itemize}
For the parameter $N$, we may suppose that its factorization is given. The precise implementation of these functions, which are crucial for the final performance, will be discussed in the next section.

\begin{algorithm}[!thbp]
  \caption{Searching for odd weird numbers up to an upper bound under $4.90 \times 10^{52}$}
  \label{algo:weird}
  \begin{algorithmic}
    \Function{Search}{$N$}
    \If{$\sigma(N) > 2N$}
    \State \texttt{CheckWeird}($N$);
    \Else
    \State $p \leftarrow$ the largest prime factor of $N$;
    \While{not \texttt{CheckBarrier}($Np$) and $Np < $ Upperbound}
    \State \Call{Search}{$Np$};
    \State $p \leftarrow$ \texttt{NextPrime}($p$);
    \EndWhile
    \EndIf
    \EndFunction

    \State \Call{Search}{3}; \Call{Search}{5}; \Call{Search}{7};
  \end{algorithmic}
\end{algorithm}

\begin{thm} \label{thm:algo-valid}
  If there is an odd weird number below the given upper bound under $4.90 \times 10^{52}$, then Algorithm~\ref{algo:weird} will find it.
\end{thm}
\begin{proof}
  We only need to show that the parts of $T$ not visited by Algorithm~\ref{algo:weird} do not contain the smallest odd weird number satisfying the given condition. First, Proposition~\ref{prop:abundant-prime-factor} ensures that all odd abundant numbers under $4.90 \times 10^{52}$ are descendants of $3$, $5$ or $7$, thus visited by one of the initial calls of the search function.

  In the recursive search function, we first note that we do not explore children of any abundant $N$. Proposition~\ref{prop:primitive-weird} ensures that we will not miss the smallest odd weird number, which must be primitively abundant. Then, we see that the while loop goes through children of $N$ in increasing order. We notice that if the condition in \ref{prop:barrier-formula} is satisfied for some $Np$ with $p$ a prime larger than the largest prime factor of $N$, then it is also satisfied for any $Np'$ with $p' > p$ a prime. This justifies the use of a while loop, ensuring that once \texttt{CheckBarrier}($Np$) returns true, no unvisited child of $N$ leads to an odd weird number satisfying the given condition. By similar reasoning, the second condition of the while loop also ensures that we do not explore any number larger than the upper bound. These are the only occasions where we stop short of exploring $T$, none of them missing out the smallest odd weird number satisfying the given condition.
\end{proof}

\begin{rmk}
  Algorithm~\ref{algo:weird} can be seen as an application of the reverse search \cite{reverse-search}, originally designed for exhaustive enumeration and combinatorial optimization, in the context of number theory. Briefly speaking, a reverse search algorithm searches the configuration space through a spanning forest. To this end, the algorithm assign a predecessor to each configuration, and when searching, it only goes from $u$ to its neighbors $v$ whose predecessor is $u$. This ensures that the search avoids multiple evaluation on the same configuration. To be efficient, an easy-to-compute predecessor function is needed. In our case, the predecessor is given by $\pred$, but we do not need to compute it, as we know how to produce all children of a given number.
\end{rmk}

\section{Implementation details} \label{sec:impl}

We have implemented Algorithm~\ref{algo:weird} to search for odd weird numbers in the ranges described in Theorem~\ref{thm:1e21}~and~\ref{thm:1e28}. We note that, to establish Theorem~\ref{thm:1e21}, by  Proposition~\ref{prop:abundant-prime-factor}, we do not need to explore the descendants of $7$ in $T$. We used the MPIR library \cite{mpir} to deal with large numbers. However, for performance, 64-bit integers are used whenever possible. To save computation time, we store $N$ along with its factorization and $\sigma(N)$. In this way, we can access the largest prime of $N$ directly, and we can also compute $\sigma(Np)$ with a constant number of operations using $\sigma(N)$.

We now detail the implementation of the three functions used in Algorithm~\ref{algo:weird}.

For \texttt{CheckBarrier}, to avoid floating point operations, we check \eqref{eq:barrier} with both sides multiplied by the denominator of the left-hand side. In practice, to reduce overhead, we do not call \texttt{CheckBarrier} for each children, but only from time to time.

For \texttt{CheckWeird}, as we store the factorization of $N$, it is easy to list all factors of $N$. We first compute all factors of $N$ not larger than $A(n)$, sort them into a list $L$, then solve the subset sum problem on $L$ with $A(n)$ as target. A naïve algorithm with a few optimizations is used. We go through elements in $L$ in decreasing order, and explore both possibilities of taking each element or not. This is done by a recursive subroutine, with the target as a parameter. In each recursive call, we ignore elements larger than the target. If the target is larger than half the sum of the remaining elements, we change the target into the sum minus the target. We then call the recursive subroutine in both cases of taking the current element or not. This solve the subset sum problem. When an abundance condition is imposed, as for Theorem~\ref{thm:1e28}, we only solve the subset sum problem for numbers satisfying the condition.

For \texttt{NextPrime}, it is clear that we never need primes larger than $2^{64} < 1.84 \times 10^{19}$. For smaller primes, we precomputed the first 5000000 primes using the sieve of Eratosthenes. For larger primes, we simply increment the number and perform the Baillie–PSW primality test implemented by Thomas R. Nicely and released into the public domain (see source code in \cite{ows-data}), which was claimed by Nicely to be deterministically correct below $2^{64}$.

As the volume of computation is large, we parallelize the search by cutting $T$ into thousands of subtrees. We then collaborated with the volunteer computing project yoyo@home, where volunteers all over the world donated computational power to work on various research projects. For details of how the parallelization worked and how we ensure correct computation, readers are referred to \cite{ows-yoyo}. We used roughly $80$ core years to establish Theorem~\ref{thm:1e21}, and roughly $150$ core years for Theorem~\ref{thm:1e28}, where one core year represent the volume of computation done by a CPU core (at that time, \textit{i.e.}, 2013--2015) in one year. The source code of the parallel application and the data is available on \cite{ows-data}.

\begin{rmk}
  Our algorithm, with its pruning strategy, can be adapted to search for other kinds of numbers with special properties related to abundance, for instance odd almost perfect numbers and others mentioned in \cite{guy}. To gain efficiency, it may be the best to perform the search simultaneously for several kinds of numbers.
\end{rmk}

\bibliographystyle{alpha}
\bibliography{abundance-boundary}

\end{document}